\documentclass[11pt,a4paper,reqno]{amsart}

\usepackage{
amsmath,
amsfonts,
amssymb,
amscd,
amsthm, 
calc, 
enumerate, 
}

\usepackage[all]{xy}

\usepackage{dsfont} 

\usepackage[utf8]{inputenc}
\usepackage[T1]{fontenc}

\newtheorem{theorem}{Theorem}[section]

\newtheorem{lemma}[theorem]{Lemma}

\newtheorem{proposition}[theorem]{Proposition}

\newtheorem{question}[theorem]{Question}
\newtheorem{remark}[theorem]{Remark}

\newtheorem{claimmain}{Claim}[]

\theoremstyle{definition}
\newtheorem{definition}[theorem]{\bf Definition}

\newcommand{\bigset}[1]{\big\{ #1 \big\}}

\renewcommand{\leq}{\leqslant}
\renewcommand{\geq}{\geqslant}

\newcommand{\vph}{\varphi} 
\newcommand{\Rat}{{\mathds Q}}
\newcommand{\Real}{{\mathds R}}
\newcommand{\NID}{\ensuremath{\mathrm{NID}}}
\newcommand{\Piz}{\Pi^{-1}}
\newcommand{\Siz}{\Sigma^{-1}}

\newcommand{\defeq}{\mathrel{\mathop:}=}

\newcommand{\defhigh}[1]{{\em #1}}

\newcommand{\naturals}{\ensuremath{\mathbb N}}

\newcommand{\acro}[1]{{\ensuremath{\mathrm{ #1}}}}

\newcommand{\kolmpref}{\ensuremath{\acro{K}}}

\begin{document}

\title{Normalized Information Distance and the Oscillation Hierarchy}

\author[K. Ambos-Spies]{Klaus Ambos-Spies}
\address[Klaus Ambos-Spies]{Universit\"at Heidelberg\\
Institut f\"ur Informatik \\ 
Im Neuenheimer Feld 205\\ 
69120 Heidelberg}
\email{ambos@math.uni-heidelberg.de}

\author[W. Merkle]{Wolfgang Merkle}
\address[Wolfgang Merkle]{Universit{\"a}t Heidelberg\\
Institut f\"ur Informatik \\ 
Im Neuen\-hei\-mer Feld 205\\ 
69120 Heidelberg}
\email{merkle@math.uni-heidelberg.de}

\author[S. A. Terwijn]{Sebastiaan A. Terwijn}
\address[Sebastiaan A. Terwijn]{Radboud University Nijmegen\\
Department of Mathematics\\
P.O. Box 9010, 6500 GL Nijmegen, the Netherlands.} 
\email{terwijn@math.ru.nl}

\date{\today}

\begin{abstract}
We study the complexity of computing the normalized information distance. We introduce a hierarchy of limit-computable functions by considering the number of oscillations. This is a function version of the difference hierarchy for sets. We show that the normalized information distance is not in any level of this hierarchy, strengthening previous nonapproximability results. As an ingredient to the proof, we demonstrate a conditional undecidability result about the independence of pairs of random strings.
\end{abstract}



\maketitle

\section{Introduction}\label{sec:introduction}
\subsubsection*{Normalized information distance}
The normalized information distance $\NID$ is a distance measure for binary strings that is based on prefix-free Kolmogorov complexity~\kolmpref. The \defhigh{normalized information distance} is defined as
\[
\NID(x,y) = \frac{E(x,y)}{\max\bigset{\kolmpref(x),\kolmpref(y)}}
\makebox[6em]{ where } 
E(x,y) = \max\bigset{\kolmpref(x|y),\kolmpref(y|x)}. 
\]
Note that $\NID$, being the ratio of two nonzero functions that are approximable from above, is computable in the limit, i.e., there is a computable rational-valued  function~$f$ 
with three arguments such that for all~$x$ and~$y$ we have  
\[
\lim_{s\rightarrow\infty} f(x,y,s) = \NID(x,y).
\]
Terwijn, Torenvliet, and~Vit\'anyi~\cite{TTV} have shown that~$\NID$ can neither be computably approximated from below nor from above, 
i.e., such a computable approximation~$f$ of~\NID\ can neither be increasing nor decreasing in~$s$. 
In particular, the function~\NID\ is not computable. In what follows, we improve on these nonapproximability results by confirming the conjecture~\cite{TTV} that for any computable approximation of~\NID, the number of oscillations is not bounded by a constant, or, equivalently, that~\NID\ is not in the oscillation hierarchy. The oscillation hierarchy is defined as the union of the classes~$\Siz_1, \Siz_2, \ldots$, where~$\Siz_k$  is the class of all functions that have a computable approximation that initially increases and switches at most~$k-1$ times between increasing and decreasing (see Section~\ref{sec:hierarchy} for formal definitions).

Related to the proof of our main result, we demonstrate that given two random strings, it is undecidable whether they are independent. In fact, this conditional undecidability result is derived in the  stronger form that there is no enumeration of pairs that includes infinitely many random pairs and where all the random pairs in the enumeration are independent. The stronger result can be viewed as a conditional immunity statement and is used in the proof of our main result.
\subsubsection*{Related work}
The concept of normalized information distance was introduced by Li et al.~\cite{Vitanyi1}, and subsequently studied in a series of papers, cf.\ Vit\'anyi~et al.~\cite{Vitanyi2} and Li and Vit\'anyi~\cite[Section 8.4]{LiVitanyi}. It has both theoretical and practical interest. While the function~\NID\ itself is noncomputable, there are computable variants that have a number of surprising practical applications. Such variants are for example defined in terms of standard compression algorithms in place of prefix-free Kolmogorov complexity.

The difference hierarchy over the computably enumerable sets, or c.e.\ sets, for short, was introduced by Ershov, cf.\ Odifreddi~\cite[IV.1.18]{Odifreddi} and Selivanov~\cite{Selivanov1995}.  It is a fine hierarchy  for the $\Delta^0_2$-sets, sometimes also referred to as the Boolean hierarchy. It can be seen as an effective version of a classical hierarchy introduced by Hausdorff, which is studied in descriptive set theory. An analogous hierarchy defined over NP is studied in complexity theory.  When restricting attention to $\{0,1\}$-valued functions, i.e., to {\em sets}, the oscillation hierarchy coincides with the difference hierarchy, as follows by the discussion following Definition~\ref{def:limit-computable}. In particular, $\Siz_1$ contains just the c.e.\ sets, $\Siz_2$ contains just the d.c.e.\ sets, i.e., the differences of c.e.\ sets, and in general $\Siz_k$ contains just  the $k$-c.e.\ sets.  These coincidences motivate the choice of our notation for the classes of the oscillation hierarchy, as the same notation has been used for the classes of the difference hierarchy, see e.g.\ Selivanov \cite{Selivanov2004}.

Recall that a hierarchy is proper if each of its levels is strictly included in the next one. Similar to the case of sets, the oscillation hierarchy is proper and does not exhaust the class of all limit-computable functions. As in the case of sets, this can be shown by elementary diagonalization arguments and, in fact, this {\em follows\/} from the analogous results for sets. Theorem~\ref{theorem:main}, our main result, asserts that $\NID$ is a natural example of a limit-computable function that is not in the oscillation hierarchy.

Note that Bennett et al.~\cite{Bennettetal} have shown that~$E$ satisfies the properties of a metric up to a constant additive term. Furthermore,~$E$ is minimal among all similar distance functions~\cite[Theorem 3.7]{Vitanyi2}. Note further that
somewhat in contrast to the definition of normalized information distance, the name information distance is used for the function~$D$ defined as
\[
D(x,y) = \min\bigset{l(p) \colon U(p,x)=y \wedge U(p,y)=x}. 
\]
Here~$U$ is the universal prefix-free machine used to define~\kolmpref. It can be shown that~$D$ and~$E$ are equal up to a logarithmic additive term~\cite[Corollary 3.1]{Vitanyi2}, i.e., we have
\[
D(x,y) = E(x,y) + O(\log E(x,y)). 
\]
\subsubsection*{Notation}
Our notation is mostly standard. For further explanations, details and background, in particular about computability theory, we refer to Odifreddi~\cite{Odifreddi} and to Downey and Hirschfeldt~\cite{DowneyHirschfeldt}.  A \defhigh{string} is a binary word, i.e., a finite sequence over the binary alphabet~$\{0,1\}$. We use \defhigh{$|x|$} to denote the \defhigh{length} of a string~$x$, the empty string~$\lambda$ is the unique string of length~$0$.  The set of strings is denoted by~$\{0,1\}^*$, the set of natural numbers is denoted by~$\omega$. The two latter sets are identified by the order isomorphism that takes the length-lexicographical ordering on the set of strings to the standard ordering on the natural numbers.

We use $\leq^+$  to denote inequality up to a fixed additive constant. For example, $f(x)\leq^+ g(x)$ means that there is a constant $c$ such that we have~$f(x) \leq g(x)+c$ for all~$x$ in some specific set that will be clear from the context. Similar notation such as~$=^+$ is defined likewise. 

Enumerations of any type of objects are always meant to be effective.
\subsubsection*{Prefix-free Kolmogorov complexity}
For further use, we compile some standard facts about Kolmogorov complexity. For proofs of these fact, as well as for definitions, details and further background, we refer to Li and Vit\'anyi~\cite{LiVitanyi} and to Downey and Hirschfeldt~\cite{DowneyHirschfeldt}. 
For a string $x$, we let $x^*$ denote the program of minimum length for $x$ that appears first in some fixed enumeration of the domain of the universal machine used to define~\kolmpref. By the latter condition, the string~$x^*$ can be computed given~$x$ and~$\kolmpref(x)$. For a string~$x$ of length~$n$, we have 
\begin{align} 
\kolmpref(x) &\leq^+ n + 2\log n, \label{eq:upperbound} \\
\kolmpref(x|\,|x|) &\leq^+ n \label{eq:lengthcode}.
\end{align}
Indeed, it holds that~$\kolmpref(x) \leq^+ n + \kolmpref(n)$ for all strings~$x$. By Chaitin's counting theorem~\cite[Theorem~3.7.6]{DowneyHirschfeldt}, 
there is a constant~$d$ such that for all~$t$ for at most a fraction of~$2^{-t+d}$ of all words~$x$ of length~$n$ we have~$\kolmpref(x) \leq^+ n + \kolmpref(n) -t$. In the special case where~$t$ is equal to~$\kolmpref(n)+1$, we obtain that at most a fraction of~$2^{-\kolmpref(n)-1+d}$ of all words~$x$ of length~$n$ is nonrandom in the sense that~$\kolmpref(x) < n$. By symmetry of information, we refer to the following chain of equations
\begin{align*} 
\kolmpref(x,y) &=^+ \kolmpref(x) + \kolmpref(y|x^*) \\
&=^+ \kolmpref(y) + \kolmpref(x|y^*),
\end{align*}
which holds for all strings~$x$ and~$y$. In case both strings have the same length, we have~$\kolmpref(xy) =^+ \kolmpref(x,y)$, and symmetry of information remains valid with~$\kolmpref(x,y)$ replaced by~$\kolmpref(xy)$. Symmetry of information is due to  Levin and G\'acs and also Chaitin~\cite[Theorem~3.10.2]{DowneyHirschfeldt}.

\subsubsection*{Outline}
The outline of the paper is as follows. First,  we review in Section~\ref{sec: limit-computable} 
limit-computable functions and introduce the oscillation hierarchy. 
Then, in Section~\ref{sec:basic-properties}, we derive some basic properties of~\NID\ and, in particular, 
reprove the known results that~\NID\ can neither be effectively approximated 
from below nor from above. Before demonstrating in Section~\ref{sec:theorem-main} 
our main result, Theorem~\ref{theorem:main}, we collect in Section~\ref{sec:some-lemmas} 
notation and facts to be used in its proof, including the already mentioned conditional immunity result, 
which is stated as Theorem~\ref{theorem:conditional-immunity}.
\section{Effective approximations and the oscillation hierarchy}\label{sec:hierarchy}
\subsubsection*{Limit-computable functions}\label{sec: limit-computable}
In this section we introduce notation that relates to approximations of real-valued functions on the natural numbers. This notation extends canonically to real-valued functions with a countable  domain like~$\{0,1\}^*$, $\Rat$, $\omega\times\omega$, or similar via the usual identification of such a domain with the set of natural numbers. In particular, this notation will be  applied to approximations of the function~\NID, which maps pairs of strings to a rational number. 

For a start, we recall the following notation from computability theory. 
\begin{definition}\label{def:limit-computable}
Let~$F\colon \omega \rightarrow \Real$ be a function. A function~$f \colon  \omega  \times\omega \rightarrow \Rat $ is an \defhigh{approximation} of~$F$, if we have for all natural numbers~$x$ that
\[
\lim_{s\rightarrow\infty} f(x,s) = F(x)
\]
The function $F$ is \defhigh{limit computable}, if~$F$ has a computable approximation. 
\end{definition}
Given a computable approximation of an $\omega$-valued function~$F$, by rounding the values of the approximation to the nearest natural number, we obtain a computable $\omega$-valued approximation~$f$ to~$F$ where then, in particular, for each argument~$x$ almost all values~$f(x,s)$ are equal to~$F(x)$.  As a consequence, $\omega$-valued limit-computable functions are just the limit-computable functions from computability theory, which are also called computably approximable or~$\Delta^0_2$-functions. By Shoenfield's Limit Lemma~\cite[IV.1.17]{Odifreddi}, such a function~$F$ is limit-computable if and only if $F$ is computable with the Halting Problem $\emptyset'$.
The three following remarks show that this equivalence is false for rational-valued functions in general but extends to rational-valued functions such as~\NID\ where for given arguments one can compute a finite set of rational numbers that contains the function value. 
\begin{remark}\label{rem:limit-lemma-false}
Let~$W_0, W_1, \ldots$ be the standard enumeration of all c.e.\ sets. Fix some enumeration~$(e_0, n_0), (e_1, n_1), \ldots$ of all pairs~$(e,n)$ such that~$n$ is in~$W_e$, and let~$W_{e,s}$ be equal to the set of all~$n_i$ such that~$e_i=e$ and~$i < s$. If we let~$F(e) = 1$ in case~$W_e$ is empty, let~$F(e)= 0$ in case~$W_e$ is infinite and, otherwise, let
\[
F(e)= 2^{-\max W_e },
\text{ then }
F(e) = \lim_{s \rightarrow \infty} f(e,s) 
\text{ where }
f(e,s) = 
2^{- \max(W_{e,s} \cup \{0\}) }.
\]
Note that~$F(e)$ is equal to~$0$ if and only~$W_e$ is infinite. The function~$f$ is computable, hence~$F$ is limit-computable. However, the function~$F$ is not computable with the halting problem because otherwise, the halting problem would decide the~$\Pi^0_2$-complete index set of all~$e$ such that~$W_e$ is infinite, a contradiction.
\end{remark}
\begin{remark}\label{rem:convergence-by-value-to-nid}
Let~$f$ be a computable approximation of~\NID, i.e., $f$ converges to~\NID\ in distance in the sense that for any arguments~$x$ and~$y$, the difference between~$f(x,y,s)$ and~$\NID(x,y)$ goes to zero. By definition of~\NID\ and the upper bounds~\eqref{eq:upperbound} and~\eqref{eq:lengthcode} on prefix-free Kolmogorov complexity, for some constant~$c$ any value of the form~$\NID(x,y)$ must be contained in the set
\[
D(x,y) = \big\{ \tfrac{i}{j} \colon i, j \le 2 (|x|+|y|) + c\big\}
\]
By rounding any value~$f(x,y,s)$ to the nearest value in~$D(x, y)$, breaking ties arbitrarily, we obtain an approximation~$f^{\mathrm{R}}$ that converges to~\NID\ not just in distance but also in value, i.e., for all~$x$ and~$y$ the approximated value~$f^{\mathrm{R}}(x,y,s)$ is equal to~$\NID(x,y)$ for almost all~$s$.
\end{remark}
\begin{remark}\label{rem:convergence-by-value}
A rational-valued function~$F$ is computable with the Halting Problem if and only if it has a computable approximation that converges to~$F$ by value in the sense of Remark~\ref{rem:convergence-by-value-to-nid}.  The proof is essentially the same as the proof of Shoenfield's Limit Lemma, details are omitted. 
\end{remark}
\subsubsection*{Increasing and decreasing phases}
Given an approximation~$f$ to some limit-comput\-able function and some argument~$x$, we consider maximum intervals of the natural numbers on which the function~$s \mapsto f(x,s)$ is increasing or is decreasing. By bounding the number of such intervals or phases from above by a constant for all arguments, we will obtain a fine hierarchy for limit-computable functions.
\begin{definition}\label{def:phases}
Let~$f$ be an approximation of some function~$\omega  \rightarrow \Real$ and fix some natural number~$x$.  With~$x$ understood, let \[
\delta(x,s) = f(x,s+1) - f(x,s)
\] 
be the \defhigh{increase of~$f$ at~$s$}, and call~$s$ \defhigh{increasing} in case~$\delta(x,s)>0$ and call~$s$ \defhigh{decreasing} in case~$\delta(x,s)<0$. 
Furthermore, a subset of the natural numbers is \defhigh{monotonic} in case it does not contain both, increasing and decreasing indices. 

For any given natural number~$x$, \defhigh{phase~$t$ of~$f$ on~$x$} is defined inductively for all~$t>0$ as follows. Phase~$1$ is equal to the maximum initial segment of~$\omega$ on which~$f$ is monotonic. In the induction step, assume that for some~$t>1$ the phases~$1$ through~$t-1$ are already defined. If the union of the latter phases is all of~$\omega$, these are the only phases of~$f$ on~$x$. Otherwise, let~$m_t$ be the maximum member in phase~$t-1$ and let phase~$t$ be equal to the maximum initial segment of~$\omega\setminus \{0,  \ldots, m_t\}$ on which~$f$ is monotonic. 

The approximation~$f$ \defhigh{reaches at most phase~$t$ on~$x$} 
in case there is no phase~$t+1$ on~$x$. In case the latter holds for all 
natural numbers~$x$, the approximation~$f$ \defhigh{reaches at most phase~$t$}. 
A phase is \defhigh{increasing} if it contains an increasing index, and a phase is \defhigh{decreasing} if it contains a decreasing index
\end{definition}
The next remark states without proofs some straightforward properties of phases.
\begin{remark}
Let~$f$ be an approximation of some function~$\omega  \rightarrow \Real$ and let~$x$ be a natural number. The phases of~$f$ on~$x$ form a partition of the natural numbers into successive contiguous intervals, which are all finite unless the partition is finite, in which case exactly the last phase is infinite. In case the function~$s \mapsto f(x,s)$ is constant, there is exactly one phase, which is neither increasing nor decreasing. Otherwise, each phase is either increasing or decreasing,  and increasing and decreasing phases alternate. With the possible exception of phase~$1$, a phase is increasing or decreasing if and only if the the least index in  the phase is increasing or decreasing, respectively. 
\end{remark}
\subsubsection*{The oscillation hierarchy}
The levels of the oscillation hierarchy introduced next stratify the class of limit-computable functions according to the number of alternations between increasing and decreasing phases.  
\begin{definition}
Let~$k$ be a nonzero natural number. A~$\Siz_k$-approximation is a computable approximation~$f$ of some function~$\naturals \rightarrow \Real$ such that on every input the first phase is increasing and~$f$  reaches at most phase~$k$. The definition of~$\Piz_k$-approximation is literally the same except that the first phase is required to be decreasing instead of increasing. 

A function~$F\colon \omega \rightarrow \Rat$ is a \defhigh{$\Siz_k$-function} in case it has a ~$\Siz_k$-approximation, the class of all $\Siz_k$-functions is denoted by~$\Siz_k$. The notion of a \defhigh{$\Piz_k$-function} and the class~$\Piz_k$ of all such functions is defined likewise. The {\em oscillation hierarchy} is defined as
\[
\bigcup_{k \ge1} \left (\Siz_k\cup\Piz_k\right )
= \bigcup_{k \ge 1} \Siz_k
= \bigcup_{k \ge 1} \Piz_k .
\]
The functions in~$\Siz_1$ and in~$\Piz_1$ are also called \defhigh{approximable from below} and \defhigh{approximable from above}, respectively.  
\end{definition}
\subsubsection*{Normalizing approximations of~\NID}
We write $\NID_s(x,y)$ for approximations of $\NID$, i.e., we have 
\[
\lim_{s\rightarrow\infty} \NID_s(x,y) = \NID(x,y)
\makebox[4em]{ and }
\NID_ s(x,y) \in \Rat .
\]
Notions relating to approximations are extended to this notation in the natural way, 
e.g., such an approximation is computable if $\NID_s(x,y)$ is a computable 
function of $s$, $x$, and~$y$. In the same fashion, let~$\kolmpref_s$ be some fixed computable 
approximation from above to~$\kolmpref_s$ with values in the natural numbers, 
and similar for conditional prefix-free Kolmogorov complexity.   
\begin{definition}\label{def:kolmogorov-approximation}
The \defhigh{Kolmogorov approximation~$\NID_s^{\kolmpref}$} to~\NID\ is defined by
\[
\NID_s^{\kolmpref}(x,y) 
 = \frac{\max\bigset{\kolmpref_s(x|y),\kolmpref_s(y|x)}}{\max\bigset{\kolmpref_s(x),\kolmpref_s(y)}} .
\]
\end{definition}
\begin{remark}\label{remark:normalized}
Let~$\NID_s$ be any effective approximation of~\NID\ that reaches at most phase~$m$ and, like in Remark~\ref{rem:convergence-by-value-to-nid}, let~$\NID_s^{\mathrm{R}}$ be the version of~$\NID_s$ where the function values have been rounded to the nearest value in the set~$D(x,y)$. Then for all~$x$ and~$y$ and for almost all~$i$, we have
\begin{equation}\label{eq:nid-r-k}
\NID_i^{\mathrm{R}}(x,y)
=\NID_i^{\kolmpref}(x,y) 
\end{equation}
because both sides of the equation converge in value to~$\NID(x,y)$  in the sense of Remark~\ref{rem:convergence-by-value-to-nid}. Let~$i_0$ be minimal such that~\eqref{eq:nid-r-k} holds with~$i$ replaced by~$i_0$, let
\[
\NID_s^{\prime}(x,y) = \NID_i^{\mathrm{R}}(x,y) \quad
\text{where~$i \le \max\{i_0,s\}$ is maximal such that~\eqref{eq:nid-r-k} holds,}
\]
and call~$\NID_s^{\prime}$ the \defhigh{normalized version} of~$\NID_s$. Note that ~$\NID_s^{\prime}$ is indeed an effective approximation of~$\NID$ and reaches at most phase~$m$, too. For a proof of the latter property, observe that~$\NID_s^{\mathrm{R}}$ reaches at most phase~$m$ since the latter function may only  
increase in~$s$ in case~$\NID_s$ increases, and a similar remark holds for decreasing. 
Thus it suffices to observe that by construction for all~$x$ and~$y$ there is a nondecreasing 
function~$g$ such that~$\NID_s^{\prime}(x,y)$ is equal to~$\NID_{g(s)}^{\mathrm{R}}$.  
\end{remark}

\section{Some basic properties of~$\NID$}\label{sec:basic-properties}
In Section~\ref{sec:theorem-main}, we will show our main result that~\NID\ is not in the oscillation hierarchy. Before, we derive in the current section some basic properties of~\NID\ and give new proofs for the known facts~\cite{TTV} that~$\NID$ is approximable from neither below nor above.
\begin{lemma} \label{lemma:limsup-nid}
The values of~$\NID(x,y)$ come arbitrarily close to~$0$ and~$1$ even if the arguments are restricted to strings~$x$ and~$y$ of the same length. In fact, the following slightly stronger assertions hold
\begin{alignat}{4}
\label{eq:nid-liminf}
\lim_{n \rightarrow \infty} \; &\max \{ \NID(x,x) &\colon |x|= n\} \;&=&\; 0,\\
\label{eq:nid-limsup}
\lim_{n \rightarrow \infty}\;  &\max \{ \NID(x,0^n) &\colon |x|=n\} \;&=& \;1.
\end{alignat}
\end{lemma}
\begin{proof}
For a proof of~\eqref{eq:nid-liminf}, observe that by definition we have~$\NID(x,x) = \frac{\kolmpref(x|x)}{\kolmpref(x)}$. For fractions of the latter form, with growing length of~$x$ the denominator goes to infinity, whereas the numerator is bounded from above by a constant, so~$\NID(x,x)$ tends to~$0$. For a proof of~\eqref{eq:nid-limsup}, observe that by a standard counting argument, for some constant~$c$, all sufficiently large~$n$ and some~$x$ of length~$n$, we have
\[
\kolmpref(0^n|x) \le  n -c \le  \kolmpref(x|0^n) 
 \makebox[5em]{ and } 
\kolmpref(0^n)\le   n  -  3 \log n\le  \kolmpref(x).
\]
So by definition of~$\NID$ and~\eqref{eq:upperbound}, it holds for almost all~$n$ and all such~$x$ that
\[
\NID(x,0^n) =  \frac{\kolmpref(x|0^n)}{\kolmpref(x)}\ge \frac{n-c}{n + 3 \log n}
\; \xrightarrow[n \rightarrow \infty]{} \; 1
\]
\end{proof}

Recall that a set is {\em immune\/} if it is infinite, but it 
does not contain an infinite c.e.\ subset. Immune sets were introduced 
by Post, and they play an important role in computability theory, 
cf.\ Odifreddi~\cite{Odifreddi}.
\begin{theorem} {\rm (B\={a}rzdi\c{n}\v{s})} \label{Barzdins}
The set~$\{x \colon \kolmpref(x) \geq \tfrac{1}{2}|x|\}$ is immune.
\end{theorem}
\begin{proof}
In case  the theorem were false, fix an enumeration of some infinite c.e.\ subset of the set under consideration.  Among all strings of length at least~$t$, let~$x_t$ be  the one that is enumerated first. There is a prefix machine with some coding constant~$c$ that outputs~$x_{4n}$ when given the string~$10^{n-1}$ as input, hence~$2n \le \kolmpref(x_{4n}) \le n + c$ for all~$n$, a contradiction.   
\end{proof}
\begin{proposition} \label{immune}
Let~$r$ be a real number where~$0<r<1$. Then the set 
$$
X_r = \bigset{(x,y) \colon |x|=|y| \text{ and } \NID(x,y)>r}
$$
is immune.
\end{proposition}
\begin{proof}
First note that $X$ is infinite by Lemma~\ref{lemma:limsup-nid}.
Now suppose for a contradiction that~$X_r$ has an infinite c.e.\ subset~$A$. For each~$n$ there are at most finitely many pairs~$(x,y)$ where~$|x|=|y|=n$, hence by taking an appropriate effective subsequence of some fixed enumeration of~$A$, 
we obtain an enumeration~$(x_0, y_0), (x_1, y_1), \ldots$ of some infinite c.e.\ subset of~$A$ where~$|x_n|< |x_{n+1}|$ for all~$n$. By the latter property and because~$x_n$ and~$y_n$ have equal length by definition of~$X_r$, the values~$\max\{\kolmpref(x_n),\kolmpref(y_n)\}$ tend to infinity, while the values~$\kolmpref(x_n|y_n)$ and~$\kolmpref(y_n|x_n)$ are both bounded from above by a fixed constant that does not depend on $n$. Consequently, the values~$\NID(x_n,y_n)$ tend to~$0$, a contradiction.
\end{proof}
\begin{theorem} {\rm (\cite{TTV})} \label{Sigma1}
$\NID$ is not approximable from below.
\end{theorem}
\begin{proof}
By Proposition~\ref{immune}, the set~$X_{1/3}$ defined there is immune. But if~$\NID$ were approximable from below, this set would be~c.e., hence could not be immune.
\end{proof}
\begin{lemma} \label{lemma:noseq}
There is no computable sequence of pairs $(x_k,y_k)$ such that 
$|x_k| = |y_k|$ and $\NID(x_k,y_k) < \tfrac{1}{k}$ for all~$k$. 
\end{lemma} 
\begin{proof}
Assume for a contradiction that there is a sequence as in the lemma. Fix a constant~$c_0$ such that for all strings~$x$ and~$y$ of equal length, the values~$\kolmpref(x)$ and~$\kolmpref(y)$ are both less than or equal to~$\kolmpref(xy)+c_0$. There is a prefix-free machine with some coding constant~$c_1$ that outputs~$x_{2k}y_{2k}$ when given the binary string~$10^{k-1}$ as input, hence ~$\kolmpref(x_{2k}y_{2k}) \le k + c_1$ for all~$k$. In summary,  we have
\[
\frac{1}{2 k} > \NID(x_{2k}, y_{2k}) \ge 
\frac{1}{\max\{\kolmpref(x_{2k}), \kolmpref(y_{2k})\}}   
\ge \frac{1}{\kolmpref(x_{2k}y_{2k})+c_0}
\ge \frac{1}{k+c_0 + c_1},
\]
a plain contradiction for all~$k > c_0 + c_1$.
\end{proof}
\begin{proposition} {\rm (\cite{TTV})} \label{Pi1}
$\NID$ is not approximable from above.
\end{proposition}
\begin{proof}
Assume for a proof by contradiction that the proposition is false. By Lemma~\ref{lemma:limsup-nid}, the values~$\NID(x,x)$ tend to~$0$, thus by dovetailing approximations from above to the values~$\NID(x,x)$ for all~$x$, for 
given~$k$ one can effectively find a string~$x_k$ such that~$\NID(x_k,x_k) < \frac{1}{k}$.
This contradicts Lemma~\ref{lemma:noseq}.
\end{proof}
\begin{theorem}
$\NID$ is not a $\Siz_2$-function.
\end{theorem}
\begin{proof}
Suppose for a contradiction that $\NID$ is $\Siz_2$, that is, it has a computable approximation that starts with an increasing phase and reaches at most phase~2. 
Consider the pairs~$(x,y)$ of words where~$|x|=|y|$. By Lemma~\ref{lemma:limsup-nid}, there are infinitely many such pairs $(x,y)$ where~$\NID(x,y)>\tfrac{3}{4}$. Consequently, we can effectively find infinitely many such pairs $(x,y)$  such that the approximation of~$\NID(x,y)$ attains a value strictly larger than~$\tfrac{3}{4}$ during phase~1. If for some~$k\ge 2$ and almost all pairs~$(x,y)$ of the latter kind it would actually hold that $\NID(x,y)>\tfrac{1}{k}$ this would contradict Proposition~\ref{immune}. 
As a consequence, for every~$k\ge 2$ there is a pair~$(x,y)$ of words of identical length where the approximation becomes smaller than~$\tfrac{1}{k}$ during phase~2, and for all such~$k$, $x$, and~$y$ we have~$\NID(x,y) < \tfrac{1}{k}$ because the approximation never reaches phase~$3$. For given~$k$ such~$x$ and~$y$ can be found effectively, which contradicts Lemma~\ref{lemma:noseq} 
\end{proof}
\section{Conditional independence} \label{sec:some-lemmas}
\subsubsection*{Random and independent pairs}
Before we demonstrate in the next section our main result, Theorem~\ref{theorem:main}, we collect some notation and facts used in its proof.
\begin{definition}\label{def:random-pairs}
Let~$r>0$ be a real number and let~$a$ and~$a'$ be words. 
The word~$a$ is \defhigh{random} if~$\mathrm{K}(a) \ge |a|$, and the pair~$(a, a')$ is \defhigh{random} if~$a$ and~$a'$ are both random. 

The string~$a$ is \defhigh{$r$-compressible} if~$\mathrm{K}(a) \le r |a|$, and the pair~$(a, a')$ is \defhigh{$r$-compressible} if~$a$ and~$a'$ are both $r$-compressible. The pair~$(a, a')$ is \defhigh{mutually $r$-compressible} if we have 
\begin{equation}\label{eq:def-mutual-incompressible}
\mathrm{K}(a|a') \le r |a|  \makebox[ 4em ]{ and }
\mathrm{K}(a'|a) \le r |a'| .
\end{equation}
\end{definition}
\begin{lemma}\label{lemma:number-of-random-pairs}
Let~$\varepsilon >0$ be a real number. For almost all~$n$, all but a fraction of at most~$\varepsilon$ of the pairs~$(a, a')$ of words of equal length~$n$ are random \end{lemma}
\begin{proof}
By Chaitin's counting theorem, there is a constant~$d$ such that for given~$n$, at most~$2^{n-\kolmpref(n)-1+d}$ many pairs have a nonrandom first component, and the same bound holds for the number of pairs with nonrandom second component. Consequently, among the~$2^{2n}$ pairs of words of length~$n$ at most~$2^{n-\kolmpref(n)+d}$ are nonrandom, which is a fraction of at most~$\varepsilon$ for almost all~$n$.
\end{proof}
Recall that an order is a function with values in the natural numbers that is nondecreasing and unbounded
\begin{definition}\label{def:dependent-pairs}
With some computable order~$h$ understood, the pair~$(a, a')$ is \defhigh{independent conditioned on a string~$x$} if we have
\begin{equation}\label{ref:def-independent-pair}
\kolmpref(a| {a'}^\ast, {x}^\ast) \ge |a| - h(n) \makebox[ 4em ]{ and }
\kolmpref(a'|a^{\ast},  {x}^\ast) \ge |a'| - h(n),
\end{equation} 
and~$(a, a')$ is \defhigh{independent} if it is independent conditioned on the empty string.
\end{definition}
\begin{lemma}\label{lemma:number-of-independent-pairs}
Let~$\varepsilon >0$ be a real number and let~$h$ be some order. Then there is some~$n_0$ such that for all~$n \ge n_0$ and for any fixed word~$x$, all but a fraction of at most~$\varepsilon$ of the pairs~$(a, a')$ of words of equal length~$n$ are independent conditioned on~$x$.  
\end{lemma}
\begin{proof}
Fix any natural number~$n$ and any word~$x$. The number of words of length strictly less than~$n-h(n)$ is bounded from above by~$2^{n-h(n)}$, hence for given~$a^{\prime}$ the latter bounds also the number of words~$a$ such that~$\kolmpref(a| {a'}^\ast, {x}^\ast) < n - h(n)$. As a consequence, the number of pairs~$(a,a')$ of words of length~$n$ that do not satisfy the first inequality in~\eqref{ref:def-independent-pair} is at most~$2^n 2^{n-h(n)}$. By symmetry, the same upper bound holds for the number of pairs~$(a,a')$ that do not satisfy the second inequality in~\eqref{ref:def-independent-pair}. Consequently,  among the~$2^{2n}$ pairs of words of length~$n$, at most~$2 \cdot 2^{2n-h(n)}$ many pairs are not independent conditioned on~$x$, i.e., at most a fraction of~$2 \cdot 2^{-h(n)}$.  The latter bound is at most~$\varepsilon$ for all~$n$ larger than some appropriate number~$n_0$ that does not depend on~$x$. 
\end{proof}
\subsubsection*{Conditional immunity}
As a further ingredient to the proof of Theorem~\ref{theorem:main}, we derive a result about the undecidability of independence of random strings. More precisely, we show that there is no algorithm that, given two random strings of the same length, can decide whether they are independent or not, where it is agreed that the algorithm may fail to converge or to give the right answer if one or both of the strings are not random. In fact, we need a stronger fact, which will be formulated in terms of the following notion of conditional immunity. 

\begin{definition}
A set $A$ is {\em decidable conditional to\/} a set~$C$ if there is a partial computable function $\vph$ such that for all~$x$ in~$C$ the value~$\vph(x)$ is defined and equal to~$A(x)$.

A set~$A$ is {\em immune conditional to\/} a set~$C$ if there is no c.e.\ set $B$ such that~$B\cap C$ is an infinite subset of~$A$. 
\end{definition}

Decidability and immunity conditional to the set of natural numbers are just classical decidability and immunity, respectively. Classically, a set is not immune if it has an infinite~c.e.\ subset, where one can always assume that this c.e.\ subset is indeed decidable, since every infinite c.e.\ set contains a decidable subset. This assumption is false in general for the conditional variant of immunity, since all decidable subsets of the considered c.e.\ set may have a finite intersection with the conditional set~$C$.

Note that if $A\cap C$ is infinite, and $A$ is decidable conditional to $C$, then $A$ cannot be immune conditional to $C$. Hence, conditional immunity is a strong form of conditional undecidability. 

\begin{theorem} \label{theorem:conditional-immunity}
Let~$r>0$ be a real number. Let~$R$ be the set of random pairs of equal length and let~$I$ be the set of pairs of equal length that are not mutually r-compressible, i.e., let
\begin{align*}
R &= \bigset{(x,y) \colon  |x|=|y| \wedge \kolmpref(x)\geq |x| \wedge \kolmpref(y)\geq |y|}, \\
I &= \bigset{(x,y) \colon  |x|=|y| \wedge \big(\kolmpref(x|y)> r|x| \vee \kolmpref(y|x) > r|y|\big)}.
\end{align*}
Then the set~$I$ is immune conditional to $R$.
\end{theorem}
\begin{proof} 
Suppose for a contradiction that there exists a c.e.\ set $B$ such that~$R\cap B$ is an infinite subset of~$I$. Fix any pair~$(x,y)$ in~$R \cap B$ and w.l.o.g.\  assume~$\kolmpref(y|x)\geq r|y|$. If we let~$n$ be equal to the length of~$x$ and~$y$, we have for some constant~$c$ 
\begin{align*}
\kolmpref(xy) &\geq^+ \kolmpref(x) + \kolmpref(y|x^*) \\
&\geq^+ \kolmpref(x) + \kolmpref(y|x,\kolmpref(x)) \\
&\geq^+ n + \kolmpref(y|x) - c \log n \\
&\geq\phantom{^+}  n + rn - c \log n.
\end{align*}
Here the inequalities follow, from top to bottom, by the variant of symmetry of information stated in the paragraph on Kolmogorov complexity, because~$x^*$ can be computed given~$x$ and~$\kolmpref(x)$, because applying~\eqref{eq:upperbound} twice yields~$\kolmpref(\kolmpref(x)) < c \log n$ for some constant~$c$, and, finally, by assumption on the pair~$(x,y)$. 

Consider any~$n$ that is so large that~$c \log n < \tfrac{r}{2}n$ and where~$R \cap B$ contains pairs~$(x,y)$ of words of length~$n$. Then, on the one hand, for each such pair, we have~$\kolmpref(xy) \ge^+ n + \tfrac{r}{2}n$. On the other hand, for each such~$n$ there is such a pair~$(x_n, y_n)$ where~$\kolmpref(x_n y_n)\leq^+ n$,  a contradiction. In order to obtain~$(x_n, y_n)$ as claimed, let~$z_n$ be the string of length~$n$ that is enumerated last in some fixed enumeration of all nonrandom strings (of all lengths).  Then knowing~$z_n$ one knows all random strings of length~$n$. Thus we can compute from $z_n$ the pair~$(x_n,y_n)$ that among all random pairs of strings of length~$n$ is enumerated first into~$B$. Since~$\kolmpref(z_n)< n$, we have $\kolmpref(x_ny_n)\leq^+ n$. \end{proof}
\section{$\NID$ is not in the oscillation hierarchy}\label{sec:theorem-main}
Our main result Theorem~\ref{theorem:main} asserts that~$\NID$ is not in the oscillation hierarchy, which confirms a conjecture by Terwijn, Torenvliet, and~Vit\'anyi~\cite{TTV}. 

We begin by giving an informal description of the proof of Theorem~\ref{theorem:main}. 
For a proof by contradiction, we assume that there is a computable approximation~$\NID_s$ to~\NID\ that reaches at most phase~$m$ for some natural number~$m$.  By Remark~\ref{remark:normalized},
we can assume that this approximation~$\NID_s$ is normalized, i.e., is obtained by approximating prefix-free Kolmogorov complexity. We may thus argue, for example, that the approximated values~$\NID_s(x,y)$ become larger in case the approximations to~$\kolmpref(x)$ and~$\kolmpref(y)$ become smaller while the approximations to~$\kolmpref(y|x)$ and~$\kolmpref(y|x)$ remain the same. By using such formulations we aim at a very rough intuitive description of the phenomena that occur, which is, however, not precise enough to provide a sketch of the formal proof. 

In the proof of Theorem~\ref{theorem:main}, we fix rational numbers~$\alpha$ and~$\beta$ where~$\beta < \alpha <1$. The proof has an inductive structure where in the induction step we consider approximations~$\NID_s(w,w^{\prime})$ for pairs of strings~$w=abc$ and~$w^{\prime}=a^{\prime}b^{\prime}c^{\prime}$ where~$a$ and~$a^{\prime}$, $b$ and~$b^{\prime}$, as well as~$c$ and~$c^{\prime}$ are of identical length, and where~$a$ has length~$n$, $b$ has length~$2n$, and~$c$ has length~$\ell n$ for some fixed~$\ell$ where~$6 \le \ell \le 3^m-3$. 
\begin{figure}[htb]
\mbox{} \hfill
\xymatrix@R=10pt@C=10pt@C=20pt@M=0pt{
\ar^*+{a}@{|-|}[r] &  
\ar^*+{b}@{-|}[rr] && 
\ar^*+{c}@{-|}[rrrrrr] &&&&&& \hspace*{.5cm} w \\
\ar^*+{a^{\prime}}@{|-|}[r] &  
\ar^*+{b^{\prime}}@{-|}[rr] &&            
\ar^*+{c^{\prime}}@{-|}[rrrrrr] &&&&&& \hspace*{.5cm} w^{\prime} 
}
\hfill \mbox{}
\end{figure}
We use an independence condition for pairs of words where the fraction of pairs  that do not satisfy the condition among all pairs of word of length~$\ell n$ tends to zero when~$n$ goes to infinity. Thus if some property holds for almost all~$n$ and a constant nonzero fraction of all pairs~$(c,c^{\prime})$ of words of length~$\ell n$, then for for almost all~$n$ and some slightly smaller constant nonzero fraction of all such pairs, both the property and the independence condition hold.   

In the induction step, we assume that there is an increasing phase~$t_0$ during which the approximation goes above~$\alpha$ for infinitely many pairs~$(a,a^{\prime})$ and some constant nonzero fraction of all pairs~$(bc,b^{\prime}c^{\prime})$. Then we argue that this includes infinitely many pairs~$(ab,a^{\prime}b^{\prime})$ such that at some later stage the pair~$(b,b^{\prime})$ appears to be at the same time random and mutually highly compressible. By the latter property and the independence condition it follows that~$\NID(abc,a^{\prime}b^{\prime}c^{\prime})< \beta$, which in turn implies that there must be a decreasing phase~$t_1> t_0$ during which the approximation goes  below~$\beta$ for infinitely many pairs~$(ab,a^{\prime}b^{\prime})$ and some constant nonzero fraction of all pairs~$(c,c^{\prime})$. Next we argue that for infinitely many of these pairs~$(ab,a^{\prime}b^{\prime})$  it turns out later that the pair~$(b,b^{\prime})$ is mutually highly compressible, which together with the independence condition implies that~$\NID(abc,a^{\prime}b^{\prime}c^{\prime}) > \alpha$. Consequently, there must be an increasing phase~$t_2> t_1$ during which the approximation goes  above~$\alpha$ for infinitely many pairs~$(ab,a^{\prime}b^{\prime})$ and a nonzero fraction of all pairs~$(c,c^{\prime})$. 

Intuitively speaking, in the induction step it is argued that there are sufficiently many argument pairs~$(abc, a^{\prime}b^{\prime}c^{\prime})$ for which the approximation~$\NID_s$ first goes above~$\alpha$ during phase~$t_0$, then goes below~$\beta$ during phase~$t_1$, and finally goes again above~$\alpha$ during phase~$t_2$. This holds because there are  sufficiently many pairs~$b$ and~$b^{\prime}$ that first appear to be random and mutually incompressible, then, second, appear to be random and mutually compressible, and, third, finally appear to be nonrandom and mutually compressible. That is, the maximum of~$\kolmpref(b)$ and of~$\kolmpref(b)^{\prime}$ and the maximum of~$\kolmpref(b | b^{\prime})$ and~$\kolmpref(b^{\prime} | b)$ appear first to be both high, second to be high and low, respectively, and, third, to be both low, where low means close to~$0$ and high means close to~$|b|$.  That such changes, which concern only the strings~$b$ and~$b^{\prime}$, result in changes of the value of~$\NID_s(abc, a^{\prime}b^{\prime}c^{\prime})$ depends on the notion of independence. For an independent pair~$(c,c^{\prime})$, the prefix-free Kolmogorov complexity of~$c$ and~$c^{\prime}$, as well as their mutual conditional prefix-free Kolmogorov complexity conditioned in addition on~$(ab)^{\ast}$ are all so close to~$|c|$ that the influence of~$c$ on a value of the form~$\NID_s(abc, a^{\prime}b^{\prime}c^{\prime})$ can be neglected compared to the influence of~$a$, $a^{\prime}$, $b$, and~$b^{\prime}$. Since in addition the two former strings are short compared to the two latter strings, the described changes in prefix-free Kolmogorov complexity relating to~$b$ and~$b^{\prime}$, though small compared to~$|c|$, are still large enough to force~$\NID_s(abc, a^{\prime}b^{\prime}c^{\prime})$ below~$\beta$ and above~$\alpha$.
\begin{theorem} \label{theorem:main}
$\NID$ is not in the oscillation hierarchy, i.e., $\NID$ is not in $\Siz_m$ for any $m\geq 1$.
\end{theorem}
\begin{proof}
For a proof by contradiction, assume that~$\NID$ is in~$\Siz_m$ for some~$m >1$, hence has a computable approximation~$\NID_s(x,y)$ that starts with an increasing phase and reaches at most phase~$m$ on all arguments. Indeed, we can assume 
\[
0= \NID_0(x,y)  < \NID_1(x,y) \le \NID(x,y)
\]
because for given~$x$ and~$y$, we can compute an upper bound for the denominator of the expression defining~$\NID(x,y)$, hence can compute a nonzero lower bound for the latter value. 
Choose the rational~$r>0$ so small that 
\[
\alpha \defeq \frac{1-5r}{1} \makebox[12em]{ is strictly larger than } \beta 
\defeq \frac{3^m-2+4r}{3^m-1} .
\]
For the scope of this proof, call a pair~$(w,w')$ of words $t$-high in case phase~$t$ is increasing and contains some~$s$ such that~$\NID_s(w,w') > \alpha$. Similarly, call the pair $t$-low in case  phase~$t$ is decreasing and contains some~$s$ such that~$\NID_s(w,w') < \beta$. Given natural numbers~$k$ and~$t$, and a real number~$\varepsilon$, let
\begin{align*}
A (k, t, \varepsilon)  = \{(a, &a') \colon
 |a|= |a'| \text{ and for a fraction of at least~$\varepsilon$ of all pairs $(u, u')$}\\
& \text{of words of equal length $(3^k-1)|a|$, the pair $(a u , a' u')$  is $t$-high} \}.
\end{align*}
Observe that all sets of the form~$A (k, t, \varepsilon)$ are empty in case~$t>m$, 
as well as in case phase~$t$ is decreasing, by choice of $\NID_s$ and by definition of~$t$-high. 

In the remainder of this proof, the notion independent conditioned on a certain word is always meant with respect to the fixed order~$h(n)=\log n$. In particular, the values~$h(n)/n$ tend to~$0$, hence for any constant~$\ell$, we have
\[
\frac{h(\ell n)}{n} = \ell \;\frac{h(\ell n)}{\ell n}  
\xrightarrow{n\rightarrow\infty} 0.
\]
\begin{claimmain}\label{claim:one}
There is some phase~$t\le m$ such that~$A (m, t, \tfrac{1}{2m})$ is infinite.
\end{claimmain}
\begin{proof}
Let~$n$ be a natural number, let~$a=0^n$ and let~$u$ and~$u'$ be any words of length~$(3^m-1)n$. Then we have
\[
\kolmpref(a u) =^+\kolmpref(u) , 
\kolmpref(a u') =^+\kolmpref( u'),
\kolmpref(a u'|a u) =^+\kolmpref(u' |u), 
\kolmpref(a u'|a u) =^+\kolmpref(u' | u), 
 \] 
where the constants hidden in the notation~$=^{+}$ do not depend on~$n$, $a$, ~$u$ or~$u'$. Thus for some constant~$d$ that is again independent of the latter four parameters, in case~$n$ is  sufficiently large and~$u$ and~$u'$ are independent, we have
\begin{align*}
\NID(au, au') &\ge 
\frac{\max\{ \kolmpref(u|u'), \kolmpref(u'|u)\} -d  }{\max\{ \kolmpref(u), \kolmpref(u')\} + d}\ge \frac{|u|-h(|u|)-d}{|u| + \kolmpref(|u|)+2d} \\ 
&
\ge \frac{|u|-2 \log |u|}{|u| + 3 \log |u|}
\ge \frac{|u|- 5 \log |u|}{|u| } > \alpha .
\end{align*}
By the preceding discussion and Lemma~\ref{lemma:number-of-independent-pairs}, for almost all~$n$ and at least half of all pairs~$(u, u')$ of words of length~$(3^m-1)n$, we have~$\NID(0^n u, 0^n u') > \alpha$, hence~$(0^n u, 0^n u')$ must be $t^{\prime}$-high for some phase~$t^{\prime}$, where~$t^{\prime} \le m$ by assumption on the approximation~$\NID_s$. Hence there must be some~$t \le m$ such that for infinitely many~$n$ for a fraction of at least~$\tfrac{1}{2m}$ of all pairs~$(u, u')$ of words of length~$(3^m-1)n$ the pair~$(0^n u, 0^n u')$ is $t$-high. For all such~$n$, the pair~$(0^n, 0^n)$ is in~$A (m, t, \tfrac{1}{2m})$.
\end{proof}
\begin{claimmain}\label{claim:two}
Let~$k$ and~$t$ be in~$\{2, \ldots, m\}$, and let~$\varepsilon >0$ be a real number such that~$A (k, t, \varepsilon)$ is infinite. Then~$A (k-1, t^{\prime}, \tfrac{\varepsilon}{4{m^2}})$ is infinite for some~$t^{\prime}\ge t+2$.
\end{claimmain}
Before we prove Claim~\ref{claim:two}, we argue that the first two claims imply the theorem.  By Claim~\ref{claim:one}, we can fix~$t \le m$ such that~$A (m, t, \tfrac{1}{2m})$ is infinite. By applying Claim~\ref{claim:two} to the latter set for at most~$\lceil \tfrac{m}{2}\rceil$ times, we obtain~$j \in \{1, \ldots, \lceil \tfrac{m}{2}\rceil\}$, $\widetilde{t}> m$, and~$\widetilde{\varepsilon}>0$ such that the set~$A (m-j, \widetilde{t}, \widetilde{\varepsilon})$ is infinite. This is a contradiction because the latter set must be empty as the approximation~$\NID_s$ is assumed to reach at most phase~$m < \widetilde{t}$. Observe that~$m- j \ge m -\lceil \tfrac{m}{2} \rceil  \ge 1$ since~$m>1$.

\bigskip
In order to demonstrate Claim~\ref{claim:two}, fix~$k$,~$t_0$ and~$\varepsilon$ as in the assumption of the claim. For the remainder of this proof, when using the letters~$a$, $b$, $c$, $w$, and $n$ with or without decoration in the same context, we always assume that we have
\[
w=abc, \quad |a|=n , \quad |b|=2n, \quad|c|= \ell n 
\quad \makebox[3em]{where } \ell = 3^k -3, \text{ i.e., } |w|=3^k n.
\]
In particular, we assume for all pairs of the form~$(a,a')$, $(ab, a'b')$, or similar that the two components of the pair have equal length. By abuse of notation, quantification over words and pairs of words involving the mentioned variable names is restricted to words of the form just described. For example, if we use the phrase for all words~$a$ and~$b$, this is meant as abbreviating the phrase for all~$n$ and all words~$a$ of length~$n$ and~$b$ of length~$2n$.
\begin{claimmain}\label{claim:three}
Let the pair~$(c,c')$ be independent conditioned on~$(ab, a'b')$ where~the pair~$(b,b')$ is random and mutually $r$-compressible. 
Then we have
\[\NID(abc, a'b'c') < \beta.\]
\end{claimmain}
\begin{proof}
The assumption of the claim implies that 
\[
\kolmpref (a b c) =^+ \kolmpref(a b) + \kolmpref(c| (ab)^{\ast}) \ge^+ 
|b| + |c|  - h(|c|) \ge^+ (\ell+2) n - h(\ell n),
\]
where the equation holds by symmetry of information, and the first inequality holds because~$b$ is random and~$(c,c')$ is independent conditioned on~$(ab, a'b')$. By symmetry, the derived lower bound also holds for~$\kolmpref (a' b' c')$. Furthermore, we have
\begin{align*}
\kolmpref (a b c | a' b' c') &\le^+  \kolmpref (a b | a' b' c') + \kolmpref (c | a' b' c') 
\le^+  \kolmpref (a b| a' b') + \kolmpref (c | n)\\  
&\le^+  |a| + r |b| + |c| =  (1 + 2r +\ell) n,
\end{align*}
where by symmetry again, this upper bound also holds for~$\kolmpref (a' b' c'| a b c)$. By the lower and upper bounds just derived, there is a constant~$d$ such that for all sufficiently large~$n$ we have 
\begin{align*}
\NID(a b c, a' b' c') &
= \frac{\max\{\kolmpref(a b c | a' b' c' ), \kolmpref(a' b' c'| abc)\}}{\max\{\kolmpref(a b c), \kolmpref(a' b' c')\}} 
\le \frac{(\ell+1+2r) n+d}{(\ell+2) n - h(\ell n) - d } \\
& = \frac{\ell+1+2r +d/n}{\ell+2  - h(\ell n)/n -d/n }
< \frac{\ell+1+3r}{\ell+2-r} 
< \frac{3^k-2+4r}{3^k- 1} \le \beta .
\end{align*}
\end{proof}
\begin{claimmain}\label{claim:four}
Let~$(c,c')$ be independent conditioned on~$(ab, a'b')$ and let the pair~$(ab, ab')$ be $r$-compressible. Then~$\NID(abc, a'b'c') > \alpha$. 
\end{claimmain}
\begin{proof}
For all sufficiently large~$n$, we have
\begin{align*}
\kolmpref (a b c) 
&\le ^+ \kolmpref(a b) + \kolmpref(c| (ab)^{\ast}) 
\le^+ r |ab| +  \kolmpref(c| n) \le^+ (3r + \ell)n\\
\kolmpref (a b c | a' b' c') &\ge^+  \kolmpref (c | (a'b')^{\ast}{c'}^{\ast} ) 
\ge^+ \ell n - h(\ell n),
\end{align*}
where by symmetry the upper bound holds also for~$\kolmpref (a' b' c')$ and the lower bound holds also for~$\kolmpref (a' b' c' | a b c)$. Similar to the proof of Claim~\ref{claim:three}, we obtain that there is a constant~$d$ such that for all sufficiently large~$n$ we have 
\begin{align*}
\NID(a b c, a' b' c') 
&\ge \frac{\ell n - h(\ell n)-d}{(\ell +3r) n+d} 
= \frac{\ell  - h(\ell n)/n -d/n }{\ell +3r+d/n} 
> \frac{\ell-r}{\ell+4r}> \frac{\ell - 5 r}{\ell} > \alpha
\end{align*}
\end{proof}
\begin{claimmain}\label{claim:five}
Infinitely many pairs~$(a b, a' b')$ where the pair~$(b, b')$ is random and mutually $r$-compressible are member of the set
\[
B_0= \{(ab, a'b') \colon  \text{$(a b c , a' b' c')$  is $t_0$-high for a fraction of at least  $\varepsilon/2$ of all $(c, c')$}\}.
\]
\end{claimmain}
\begin{proof}
For any pair~$(a, a')$ in~$A(k, t_0, \varepsilon)$, the pair~$(ab, a'b')$ is in~$B_0$ 
for a fraction of at least~$\varepsilon/2$ of all pairs~$(b, b')$. Otherwise, if this fraction were~$q < \varepsilon/2$, the fraction of pairs~$(bc,b'c')$ such that~$(abc, a'b'c')$ is $t_0$-high would be strictly less than $q + (1- q) \tfrac{\varepsilon}{2} < \varepsilon$, contrary to the definition of~$A(k,t_0,\varepsilon)$. By Lemma~\ref{lemma:number-of-random-pairs}, for almost all~$n$ all but a fraction of~$\varepsilon/4$ of all pairs~$(b,b')$ are random. So for almost all of the infinitely many~$(a,a')$ in~$A(k, t_0, \varepsilon)$, there is some~$(ab, a'b')$ in~$B_0$ where the pair~$(b,b')$ is random. 

For given~$s$, $w$ and~$w'$, one can compute the value of~$\NID_s(w,w')$ and the phase in which~$s$ is, hence the set~$B_0$ is c.e.\ But then the set of all pairs~$(b,b')$ such that~$(ab, a'b')$ is in~$B_0$ for some words~$a$ and~$a'$ is also c.e.\ By the discussion in the last paragraph, the latter c.e.\ set contains infinitely many random pairs, and then infinitely many of these random pairs must be mutually $r$-compressible by Theorem~\ref{theorem:conditional-immunity}.  
\end{proof}
\begin{claimmain}\label{claim:six}
There is a decreasing phase~$t_1> t_0$ such that the set~$B_1$ is infinite, where
\[
B_1 = \{(ab, a b') \colon  \text{$(a b c , a' b' c')$  is $t_1$-low for a fraction of at least  $\tfrac{\varepsilon}{3m}$ of all } (c, c')\}.
\]
\end{claimmain}
\begin{proof}
By Lemma~\ref{lemma:number-of-independent-pairs}, for almost all~$n$ and for any given words~$a$, $a'$, $b$, and~$b'$, all but a fraction of~$\varepsilon/6$ of the pairs~$(c, c')$ are independent conditioned on~$(ab, a'b')$, hence almost all pairs~$(a b , a'b')$ in~$B_0$ are a member of the set 
\begin{align*}
B'_0 = \{(a b , a'b') \colon  &\text{$(a b c , a' b' c')$  is $t_0$-high and~$(c, c')$ is independent conditioned }\\
&\text{on~$(ab, a'b')$ for a fraction of at least $\varepsilon/3 $ of all pairs } (c, c')\}.
\end{align*}
The definition of the set~$B'_0$ is meant such that the conditions on being $t_0$-high and being independent must be satisfied simultaneously for the specified fraction of all pairs~$(c,c')$, and this convention is extended to subsequent similar formulations. Since the set~$B_0$ is a subset of~$B'_0$ except for finitely many members of~$B_0$, Claim~\ref{claim:five} holds with~$B_0$ replaced by~$B'_0$. Thus by 
Claim~\ref{claim:three}, the set 
\begin{align*}
B''_0 = \{(a b , a'b') \colon  &\text{$(a b c , a' b' c')$  is $t_0$-high and~$\NID(a b c , a' b' c') < \beta$}\\
&\text{for a fraction of at least $\varepsilon/3 $ of all pairs } (c, c')\}
\end{align*}
is infinite, too. The phase~$t_0$ is increasing, hence for a pair~$(a b c , a' b' c')$ that is $t_0$-high but has $\NID$-value of less than~$\beta$, i.e., where we have
\[
\NID(abc,a'b'c')< \beta < \alpha <  \NID_{t_0}(abc,a'b'c'),
\]
there must be some decreasing phase~$t > t_0$ such that the pair is $t$-low. Since by assumption there are at most~$m$ phases, the claim follows.  
\end{proof}
\begin{claimmain}\label{claim:seven}
There is a decreasing phase~$t_2> t_1$ such that the set~$B_2$ is infinite, where
\begin{align*}
B_2 &= \{(ab, a b') \colon  \text{$(a b c , a' b' c')$  is $t_2$-high for at least $\tfrac{\varepsilon}{4{m^2}}$ of all (c, c')}\}.
\end{align*}
\end{claimmain}
\begin{proof}
The proof is very similar to the proof of Claim~\ref{claim:six} and we omit details that are obvious by this similarity. The set~$B_1$ is infinite and c.e.  Then~$B_1$ must contain infinitely many pairs  that are $r$-compressible because otherwise for almost all pairs~$(ab, a' b')$ in~$B_1$ we would have~$\kolmpref(aba' b') \ge \tfrac{r}{3} |aba' b'|$. The latter contradicts a straightforward variant of Theorem~\ref{Barzdins}, which follows by essentially the same proof as the theorem.

Furthermore, by essentially the same argument as in the case of~$B_0$ and~$B'_0$ it follows that almost all pairs in~$B_1$ are also in the set
\begin{align*}
B'_1 = \{(a b , a'b') \colon  &\text{$(a b c , a' b' c')$  is $t_1$-low and~$(c, c')$ is independent conditioned }\\
&\text{on~$(ab, a'b')$ for a fraction of at least $\tfrac{\varepsilon}{4{m}}$ of all pairs } (c, c')\}.
\end{align*}
By the preceding discussion, the set~$B'_1$ contains infinitely many pairs~$(a b', a b')$ where~$(a b c , a' b' c')$  is $t_1$-low and the assumption of Claim~\ref{claim:four} is satisfied for a fraction of at least~$\varepsilon/4m$ of all pairs~$(c,c')$, hence the set 
\begin{align*}
B''_1 = \{(a b , a'b') \colon  &\text{$(a b c , a' b' c')$  is $t_1$-low and~$\NID(a b c , a' b' c') > \alpha$}\\
&\text{for a fraction of at least $\tfrac{\varepsilon}{4{m}}$ of all pairs } (c, c')\}
\end{align*}
is infinite, too. The phase~$t_1$ is decreasing, hence for a pair~$(a b c , a' b' c')$ that is $t_1$-low but has $\NID$-value greater than~$\alpha$, there must be some increasing phase~$t_2  > t_1$ such that the pair is $t_2$-high. By assumption there are at most~$m$ phases, the claim follows.  
\end{proof}
Now Claim~\ref{claim:two} follows because we have~$t_2 > t_1 > t_0>0$ and  the set~$B_2$ is equal to~$A (k-1, t_2, \tfrac{\varepsilon}{4{m^2}})$ since we have~$|ab|=3n$ and~$|c|= (3^k-3) n = (3^{k-1}-1) 3n$.
\end{proof}
From the proof of Theorem~\ref{theorem:main} it is obvious that the examples of pairs 
of strings $x$,$y$ forcing the changes in the approximation $\NID_s$ are of rather 
long length. It would be interesting to have a more careful analysis of these 
lengths.

\begin{question}
Relate the number of oscillations of approximations of $\NID(x,y)$ to 
the length of $x$ and $y$. 
\end{question}

\end{document}